\documentclass[a4paper,11pt]{amsart}
\usepackage{amsfonts}
\usepackage{amsthm}
\usepackage{amssymb}
\usepackage{amsmath}
\usepackage{comment}
\usepackage{a4wide}
\usepackage{dsfont}
\usepackage{bbm}
\usepackage{xparse,mathrsfs,mathtools}
\usepackage{enumitem}
\usepackage{color}
\usepackage{tikz}
\usepackage{caption}

\newtheorem{thm}{Theorem}
\newtheorem{lemma}[thm]{Lemma}
\newtheorem{cor}[thm]{Corollary}

\allowdisplaybreaks

\begin{document}

\title[Diophantine condition in the LIL for lacunary systems]{Diophantine conditions in the law of the iterated logarithm for lacunary systems}
\author{Christoph Aistleitner}
\author{Lorenz Fr\"uhwirth}
\author{Joscha Prochno}

\address{Graz University of Technology, Institute of Analysis and Number Theory, Steyrergasse 30, 8010 Graz, Austria}
\email{aistleitner@math.tugraz.at} 

\address{Graz University of Technology, Institute of Analysis and Number Theory, Kopernikusgasse 24, 8010 Graz, Austria}
\email{fruehwirth@math.tugraz.at}

\address{University of Passau, Faculty of Computer Science and Mathematics, Dr.-Hans-Kapfinger-Stra{\ss}e 30, 94032 Passau, Germany}
\email{joscha.prochno@uni-passau.de}

\subjclass[2020]{Primary 42A55, 60F15; Secondary 11D04, 11D45}
\keywords{Lacunary trigonometric sums, law of the iterated logarithm, Diophantine equations}

\begin{abstract} 
It is a classical observation that lacunary function systems exhibit many properties which are typical for systems of independent random variables. However, it had already been observed by Erd\H os and Fortet in the 1950s that probability theory's limit theorems may fail for lacunary sums $\sum f(n_k x)$ if the sequence $(n_k)_{k \geq 1}$ has a strong arithmetic ``structure''. The presence of such structure can be assessed in terms of the number of solutions $k,\ell$ of two-term linear Diophantine equations $a n_k - b n_\ell = c$. As the first author proved with Berkes in 2010, saving an (arbitrarily small) unbounded factor for the number of solutions of such equations compared to the trivial upper bound, rules out pathological situations as in the Erd\H os--Fortet example, and guarantees that $\sum f(n_k x)$ satisfies the central limit theorem (CLT) in a form which is in accordance with true independence. In contrast, as shown by the first author, for the law of the iterated logarithm (LIL) the Diophantine condition which suffices to ensure ``truly independent'' behavior requires saving this factor of logarithmic order. In the present paper we show that, rather surprisingly, saving such a logarithmic factor is actually the optimal condition in the LIL case. This result reveals the remarkable fact that the arithmetic condition required of $(n_k)_{k \geq 1}$ to ensure that $\sum f(n_k x)$ shows ``truly random'' behavior is a different one at the level of the CLT than it is at the level of the LIL: the LIL requires a stronger arithmetic condition than the CLT does.
\end{abstract}
	
\maketitle

\section{Introduction and main result}

\vskip 5mm

The classical Hartman--Wintner law of the iterated logarithm (LIL) was proved by Philip Hartman and Aurel Winter in 1941 \cite{HW1941} and quantifies the typical fluctuation of sums of independent and identically distributed (i.i.d.) random variables on the scale between the central limit theorem (CLT) and the law of large numbers (LLN). More precisely, the LIL states that for a sequence $X_1,X_2,\dots$ of i.i.d.\ random variables of zero mean and finite variance $\sigma^2\in(0,\infty)$,
\begin{equation} \label{hartman_wintner}
    \limsup_{N\to\infty} \frac{\left| \sum_{k=1}^N X_k \right|}{\sqrt{2N\log\log N}} = \sigma \quad\text{almost everywhere (a.e.)}.
\end{equation}
Today it is a well-known fact in analysis and probabilistic number theory that the asymptotic behavior of sums of i.i.d.\ random variables is echoed in many ways by lacunary trigonometric sums $\sum_{k=1}^N\cos(2\pi n_k x)$ under the so-called 
Hadamard gap condition 
\begin{equation}\label{had}
\frac{n_{k+1}}{n_k} \geq q > 1, \qquad k \in\mathbb N,
\end{equation}
for a sequence $(n_k)_{k\geq 1}$ of natural numbers; this must be seen in consideration of the fact that the random variables $X_k(x):=\cos(2\pi n_k x)$ on the probability space $[0,1]$ with Borel $\sigma$-algebra endowed with Lebesgue measure $\lambda$ are identically distributed and uncorrelated (if all $n_k$, $k\in\mathbb N$, are distinct), but \emph{not} stochastically independent. Here and in all that follows, the statements remain true if cosine is replaced by sine.\\

It was shown by Erd\H{o}s and G\'al \cite{EG1955} that under the gap condition \eqref{had},
  \begin{equation} \label{erd-gal}
    \limsup_{N\to\infty} \frac{\left| \sum_{k=1}^N \cos(2\pi n_kx) \right|}{\sqrt{2 N\log\log N}} = \frac{1}{\sqrt{2}} \quad \text{a.e.},
  \end{equation}
i.e., lacunary trigonometric sums satisfy a Hartman--Wintner LIL under the Hadamard gap condition. Note that the variance of $\cos(2\pi n_k\cdot)$ is $\frac{1}{2}$, so \eqref{erd-gal} is in perfect accordance with \eqref{hartman_wintner}. Regarding normal fluctuations, Salem and Zygmund proved in \cite{SalemZyg1947} and \cite{SalemZyg1950} that under the Hadamard gap condition \eqref{had}, for every $t\in\mathbb R$, lacunary trigonometric sums satisfy the CLT
  \[
    \lim_{N\to\infty} \lambda \left( \left\{ x\in[0,1]\,:\, \frac{\sum_{k=1}^N\cos(2\pi n_kx)}{\sqrt{N/2}} \leq t \right\} \right) = \frac{1}{\sqrt{2\pi}}\int_{-\infty}^t e^{-y^2/2}\,dy =: \Phi(t);
  \]
again this is in perfect accordance with the CLT for truly independent systems. The analogy between lacunary trigonometric sums and truly random systems goes much further, as the almost sure invariance principles of Berkes \cite{berkes} and Philipp and Stout \cite{phil_st} show. Concerning large deviation principles, it was shown only recently \cite{AGKPR2023} (and also \cite{FJP2022}) that while under the large gap condition $n_{k+1}/n_k \to \infty$ the behavior of lacunary trigonometric sums is in perfect accordance with the truly independent case, under the mere Hadamard gap condition \eqref{had} surprising phenomena occur which reflect the particular arithmetic structure of the sequence $(n_k)_{k \geq 1}$.  Accordingly, while the CLT and LIL (and other results in a regime close to normal deviations) hold for lacunary trigonometric sums in a universal form, the large deviation behavior of lacunary trigonometric sums is very sensitive to fine arithmetic properties of the sequence $(n_k)_{k\geq 1}$. This is a very interesting effect, which is currently not well understood. On which deviation scale do fine arithmetic phenomena start to play a crucial role for the probabilistic theory of lacunary trigonometric sums? A first step towards a resolution of this question has been taken very recently by the last author together with Strzelecka \cite{PS2023}, who studied moderate deviations principles (MDPs) for lacunary trigonometric sums; recall that MDPs cover the deviation range between the CLT and a large deviations principle.\\

The results discussed so far all concern the case of ``pure'' trigonometric sums $\sum \cos (2 \pi n_k x)$ or $\sum \sin (2 \pi n_k x)$. It turns out that for more general lacunary sums $\sum f(n_k x)$ with a 1-periodic function $f$ the heuristics that ``lacunary sums mimic the behavior of sums of independent random variables'' remains largely intact, but the situation becomes much more delicate. We assume that $(n_k)_{k \geq 1}$ satisfies the Hadamard gap condition, and that
$f:\mathbb{R} \to \mathbb{R}$ is a function satisfying
\begin{equation} \label{f_equ}
f(x+1) = f(x), \qquad \int_0^1 f(x) ~dx = 0, \qquad \textup{Var}_{[0,1]} f < \infty,
\end{equation}
where $\textup{Var}_{[0,1]}$ denotes the total variation of $f$ on the interval $[0,1]$ (note that bounded variation implies integrability). A crucial observation in this setup is the Erd\H os--Fortet example (see, e.g., \cite[p.646]{Kac49}) of the sequence $n_k = 2^k-1$, $k \in\mathbb N$, and the periodic function $f(x) = \cos(2 \pi x) + \cos(4 \pi x)$, for which the classical CLT for $\sum f(n_k x)$ fails to hold (and the distribution of the normalized sums instead converges to a ``variance mixture'' Gaussian distribution), and for which a ``non-standard'' LIL holds in the form
  \begin{equation} \label{erd-fort}
    \limsup_{N\to\infty} \frac{\left| \sum_{k=1}^N f(n_k x) \right|}{\sqrt{2 N\log\log N}} = \left| 2 \cos(\pi x) \right| \quad \text{a.e.}
  \end{equation}
with a (non-constant) function on the right-hand side. Thus, for $\sum f(n_k x)$ the LIL can fail to hold in its truly independent form, and instead as a general result we only have an upper-bound LIL
\begin{equation*}
    \limsup_{N\to\infty} \frac{\left| \sum_{k=1}^N f(n_k x) \right|}{\sqrt{2 N\log\log N}} \leq c_{f,q} \quad \text{a.e.},
\end{equation*}
with a constant $c_{f,q} \in (0, \infty) $ depending on $f$ and the growth factor $q$ of $(n_k)_{k \geq 1}$; see Takahashi \cite{taka} for this result, and Philipp \cite{philipp} for a generalization to the so-called Chung--Smirnov type LIL. The interaction of analytic, arithmetic and probabilistic effects that underpins this theory has led to a wealth of research, leading from famous classical papers such as those of Kac \cite{kac} and Gaposhkin \cite{G1966} to recent work such as that of Berkes, Philipp and Tichy \cite{bpt}, Bobkov and G\"otze \cite{bg}, Conze and Le Borgne \cite{clb}, and in particular Fukuyama \cite{fuku2,fuku1,fuku3,f4}.\\

An interesting observation is that in the general framework, the fine probabilistic behavior of lacunary sums $ \sum f(n_k x)$ is intimately related to the number of solutions of certain linear Diophantine equations, such as the two-variable equation 
\begin{equation} \label{dio_equ}
a n_k - b n_\ell = c.
\end{equation}
Here $a,b \in \mathbb{N}$ and $c \in \mathbb{Z}_{\geq 0 }$ are fixed, and one has to consider the number of solutions $(k,\ell)$ of the equation with the size of the indices $k,\ell$ being bounded above by some threshold value. If $N \in \mathbb{N}$, we shall write 
\begin{equation} \label{L_def}
L(N,a,b,c) := \# \left\{1 \leq k,\ell \leq N:~a n_k - b n_\ell = c \right\}.
\end{equation}
We restrict ourselves to non-negative integers $c$, since we can always switch to this case by exchanging the roles of the parameters $k$ and $\ell$ and that of $a$ and $b$, respectively.
Note that trivially $L(N,a,b,c) \leq N$ for any $a,b,c$ with $(a,b,c) \neq (0,0,0)$ and any $N\in\mathbb N$, as long as $(n_k)_{k \geq 1}$ is a sequence of distinct integers. In \cite{AB2010} it was proved that $\sum f(n_k x)$ satisfies the CLT under the assumption that the number of solutions to Diophantine equations of the form \eqref{dio_equ} is asymptotically less than the trivial estimate. More precisely, it was shown in \cite[Theorem 1.1]{AB2010} that for any function as in \eqref{f_equ} and any lacunary sequence $(n_k)_{k \geq 1}$, and any $t \in \mathbb{R}$,
$$
\lambda \left( \left\{ x \in [0,1]\,:\,\sigma_N^{-1} \sum_{k=1}^N f(n_k x) \leq t \right\} \right) \to \Phi(t) \qquad \text{as $N \to \infty$}, 
$$
with 
$$
\sigma_N^2 := \int_0^1 \left( \sum_{k=1}^N f(n_k x) \right)^2 \,dx,
$$
provided that the following two conditions are satisfied:
\begin{enumerate}[label=(\roman*)]
\item The limiting variance is not degenerate, i.e., $\sigma_N^2 \geq C N$ for some suitable constant $C > 0$.
\item For all positive integers $a,b$ with $a \neq b$, the number of solutions to the Diophantine equation satisfies\footnote{The case $a=b$ is not relevant for our paper, since when $a=b$ and $(n_k)_{k \geq 1}$ is lacunary then by Lemma \ref{zyg_lemma} below we always have $L(N,a,a,c)=\mathcal{O}(1)$ uniformly in $c$ for $c \neq 0$, and consequently the number of solutions is small enough to be negligible. For $a=b$ and $c=0$ we trivially always have $L(N,a,a,0) =N$ (these solutions are the ``diagonal terms'' and cannot be avoided).} 
\begin{equation*} 
L(N,a,b,c) = o(N) \qquad \text{uniformly in $c \in \mathbb{Z} \backslash \{0\}$}. 
\end{equation*}
\end{enumerate}

Let us remark that the condition (i) on the non-degeneracy of the variance already appeared in the work of Gaposhkin \cite{G1966} and is indeed necessary, as shown by examples leading to telescoping sums such as $f(x) = \cos(2 \pi x) - \cos(4 \pi x)$ and $n_k = 2^k,~k\geq 1$. As proved in \cite{AB2010}, the Diophantine condition $L(N,a,b,c) = o(N)$ is optimal and cannot be replaced by $L(N,a,b,c) \leq \varepsilon N$ for a fixed $\varepsilon>0$. If additionally to (ii) the sequence $(n_k)_{k \geq 1}$ satisfies $L(N,a,b,0)=o(N)$ for all $a \neq b$, that is, if the number of solutions of \eqref{dio_equ} with $c=0$ on the right-hand side is also small, then (i) is not necessary since in that case one has $\sigma_N = \|f\|_2 \sqrt{N}$ as $N \to \infty$, and the CLT holds with exactly the same normalizing factor as in the truly independent case (see \cite[Theorem 1.2.]{AB2010}). Note that this discussion also explains why the CLT fails to hold for the Erd\H os--Fortet example: for the sequence $n_k =2^k-1,~k \geq 1,$ there are too many solutions to the equation
$$
n_k - 2 n_\ell = 1,
$$
namely all $N-1$ pairs $(k,\ell)$ of the form $k= \ell+ 1$ (cf.\ Equation \eqref{erd-fort-2} below, which explains how the function $2 \cos (\pi x)$ on the right-hand side of \eqref{erd-fort} arises).
\\

As explained in the previous paragraph, the results in \cite{AB2010} provide optimal Diophantine conditions guaranteeing the CLT for $\sum f(n_k x)$. Thus, the relation between sums of dilated functions and arithmetic information in form of the number of solutions of Diophantine equations is completely understood at the level of the CLT. In contrast, the situation in the case of the LIL is much less satisfactory. It was proved in \cite[Theorem 1.3]{A2010} that for $f$ as in \eqref{f_equ} and $(n_k)_{k \geq 1}$ as in \eqref{had}, 
\begin{equation} \label{LIL_rand}
\limsup_{N \to \infty} \frac{\left| \sum_{k=1}^N f(n_k x) \right|}{\sqrt{2 N \log \log N}} = \|f\|_2 \qquad \text{a.e.},
\end{equation}
provided that for all fixed positive integers $a,b$ with $a \neq b$,
\begin{equation} \label{dio_LIL}
L(N,a,b,c) = \mathcal{O} \left( \frac{N}{(\log N)^{1+ \varepsilon}} \right), \qquad \text{uniformly in $c \in \mathbb{Z}_{\geq 0}$},
\end{equation}
for some constant $\varepsilon>0$. Equation \eqref{LIL_rand} is in perfect accordance with truly independent behavior. However, unlike in the CLT case, it was unclear whether the Diophantine condition \eqref{dio_LIL} for the LIL case was optimal. There were good reasons to believe that the factor $(\log N)^{1 + \varepsilon}$ in the stronger Diophantine condition \eqref{dio_LIL} is an artifact coming from the particular proof strategy in \cite{A2010}, which as a key ingredient evokes a classical almost sure invariance principle (ASIP) of Strassen \cite{strassen}. Roughly speaking, Strassen's ASIP for martingale differences requires the almost sure convergence of conditional second moments, which can essentially be established from \eqref{dio_LIL} using Chebyshev's inequality. Such an argument seems rather wasteful, and hence some effort was put into trying to relax the Diophantine condition for the LIL down to the one which is known to be sufficient in the CLT case. However, in the present paper we prove the rather surprising result that the Diophantine condition \eqref{dio_LIL} is actually optimal (up to lower-order terms) to ensure the LIL for $\sum f(n_k x)$, even when $f$ is restricted to be a trigonometric polynomial. Our main result is the following.

\begin{thm} \label{th1}
Let $\varepsilon \in (0,1)$. Then, for every constant $K\in(0,\infty)$, there exist a trigonometric polynomial $f$ with mean zero and a lacunary sequence $(n_k)_{k \geq 1}$ such that for all $a,b \in \mathbb{N}$ with $a \neq b$, we have
\begin{equation} \label{dio_LIL_2}
L(N,a,b,c) = \mathcal{O} \left( \frac{N}{(\log N)^{1 - \varepsilon}} \right), \qquad \text{uniformly in $c \in \mathbb{Z}_{\geq 0 }$},
\end{equation}
and such that
$$
\limsup_{N \to \infty} \frac{\left| \sum_{k=1}^N f(n_k x) \right|}{\sqrt{2 N \log \log N}} \geq K \|f\|_2 \qquad \text{a.e.}
$$
\end{thm}

Theorem \ref{th1} is remarkable as it shows that to guarantee the validity of a probabilistic limit theorem for lacunary sums of dilated functions on the LIL scale, one needs stronger arithmetic assumptions than one does on the CLT scale. We consider this to be a very interesting phenomenon. The necessary savings factor of order roughly $\log N$ in the Diophantine condition for the LIL seems to arise essentially as $e^{(\sqrt{2 \log \log N})^2 / 2}$ from the order of the tail of the normal distribution, and thus be directly connected with the fact that the LIL is concerned with deviations exceeding the CLT normalization $\sqrt{N}$ by an additional factor $\sqrt{2 \log \log N}$. One cannot help but wonder if a similar direct connection between the necessary arithmetic (Diophantine) condition and the size of the deviation that one is interested in  persists throughout other scales; note that such a direct link would have to become meaningless at least for additional factors of order exceeding $\sqrt{2 \log N}$, which would correspond to the requirement of saving a factor of more than $e^{(\sqrt{2 \log N})^2/2} = N$ in the Diophantine condition, thus asking for less than one solution in \eqref{L_def}, which is absurd.\footnote{Possibly it is no coincidence that martingale methods based on conditional second moments also seem to reach a critical point at deviations of order $\sqrt{2 \log N}$, see for example \cite{grama}.} 
If so, then the arithmetic theory underpinning the behavior of lacunary sums at small-scale deviations near $\sqrt{N}$ would be substantially different from the corresponding theory for deviations at large scales beyond $\sqrt{2 N \log N}$. One possible explanation for such a dichotomy could be that two-term Diophantine equations can only control the distribution of normalized lacunary sums at small deviation scales, and that at larger scales a different effect sets in which is only expressible in terms of Diophantine equations in more than 2 variables. We leave these questions for future research.

\vskip 5mm
\section{Construction of the sequence \& Solutions to Diophantine equations} \label{sec_construction}

We now present a completely explicit construction of a Hadamard gap sequence $(n_k)_{k \geq 1}$ satisfying the claim of Theorem \ref{th1}, and we show that it indeed satisfies the bound \eqref{dio_LIL_2} for the number of solutions of the two-variable linear Diophantine equations in \eqref{dio_equ}.\\

\underline{Step 1.} Let $\varepsilon \in (0,1)$ and $K \in (0,\infty)$ be given. Choose $d \in \mathbb{N}$ such that
\begin {equation} \label{size_d}
\frac{d \sqrt{\varepsilon}}{4}-2 > \frac{K \sqrt{d}}{\sqrt{2}};
\end{equation}
note that this condition can be satisfied by choosing $d$ sufficiently large (with the necessary size of $d$ depending on the parameters $\varepsilon$ and $K$). The number $d$ will later be the degree of the trigonometric polynomial $f$, which we construct in order to prove Theorem \ref{th1}. Further, let $R := R(\varepsilon) \in \mathbb{N}$ such that
\begin{equation} \label{size_R}
R > \frac{8}{\varepsilon}.
\end{equation}
The parameter $R$ will serve as decomposition parameter; we split the set of positive integers $\mathbb{N}$ into consecutive blocks $\Delta_1, \Delta_2, \dots$ such that $\# \Delta_i = R^i$, i.e., the sizes of the block are rapidly increasing. More precisely, we define
\begin{equation} \label{delta_def}
\Delta_i := \left\{\frac{R^i - R}{R-1} +1, \dots, \frac{R^{i+1}-R}{R-1} \right\}, \qquad i \in\mathbb N.
\end{equation}
From this construction it follows that for each $i \in\mathbb N$, 
\begin{equation*} 
\sum_{h=1}^{i-1} \# \Delta_h \leq \frac{1}{R-1} \# \Delta_i,
\end{equation*} 
i.e., the block $\Delta_i$ is even much larger than the collection of all previous $i-1$ blocks taken together.  
To put it more illustratively, the partial sum $\sum_{k \in \Delta_1 \cup \dots  \cup \Delta_i} f(n_k x)$ will be dominated by the terms with $k \in \Delta_i$, while the terms with $k \in \Delta_1 \cup \dots \cup \Delta_{i-1}$ will be essentially negligible, so that the lower bound in Theorem \ref{th1} only has to be established for sums $\sum_{k \in \Delta_i} f(n_k x)$ as $i \to \infty$. \\

\underline{Step 2.} We shall now split up each block $\Delta_i$, $i \in\mathbb N$, into disjoint subsets, where the number of subsets depends on the parameter $i$. More precisely, we decompose $\Delta_i$ into\footnote{Throughout the paper $\lceil x \rceil$ denotes the smallest integer which is at least as large as $x$.}
$$
\Delta_i^{(m)}, \qquad 1 \leq m \leq M(i) :=\lceil i^{1-\varepsilon} \rceil, 
$$
such that $\Delta_i^{(1)} < \Delta_i^{(2)}< \dots < \Delta_i^{(M(i))}$ holds element-wise with $\Delta_i = \cup_{m=1}^{M(i)} \Delta_i^{(m)}$, and such that all sets $\Delta_i^{(m)},~1 \leq m \leq M(i),$ have essentially the same cardinality, i.e.,
\begin{equation} \label{mi_card}
\left| \# \Delta_i^{(m)} - \frac{R^i}{\lceil i^{1-\varepsilon} \rceil} \right| \leq 1, \qquad 1 \leq m \leq M(i).
\end{equation}
Heuristically, this construction is made in such a way that if we write $\Delta_1 \cup \dots \cup \Delta_i = \{1, \dots, N\}$ for some suitable $N \in \mathbb{N}$, then 
\begin{equation} \label{delta_i_m_heuristic}
\# \Delta_i^{(m)} = \mathcal{O} \left( \frac{N}{(\log N)^{1-\varepsilon}} \right) \qquad \text{for all $m \in \{1, \dots, M(i)\}$},
\end{equation}
which reflects our bound \eqref{dio_LIL_2} on the number of solutions of Diophantine equations; note that the implied constant in the $\mathcal{O}$-term depends on the parameter $R$.\\

\begin{figure}[h!]
\centering
 \includegraphics[trim=80 670 80 95, scale=1]{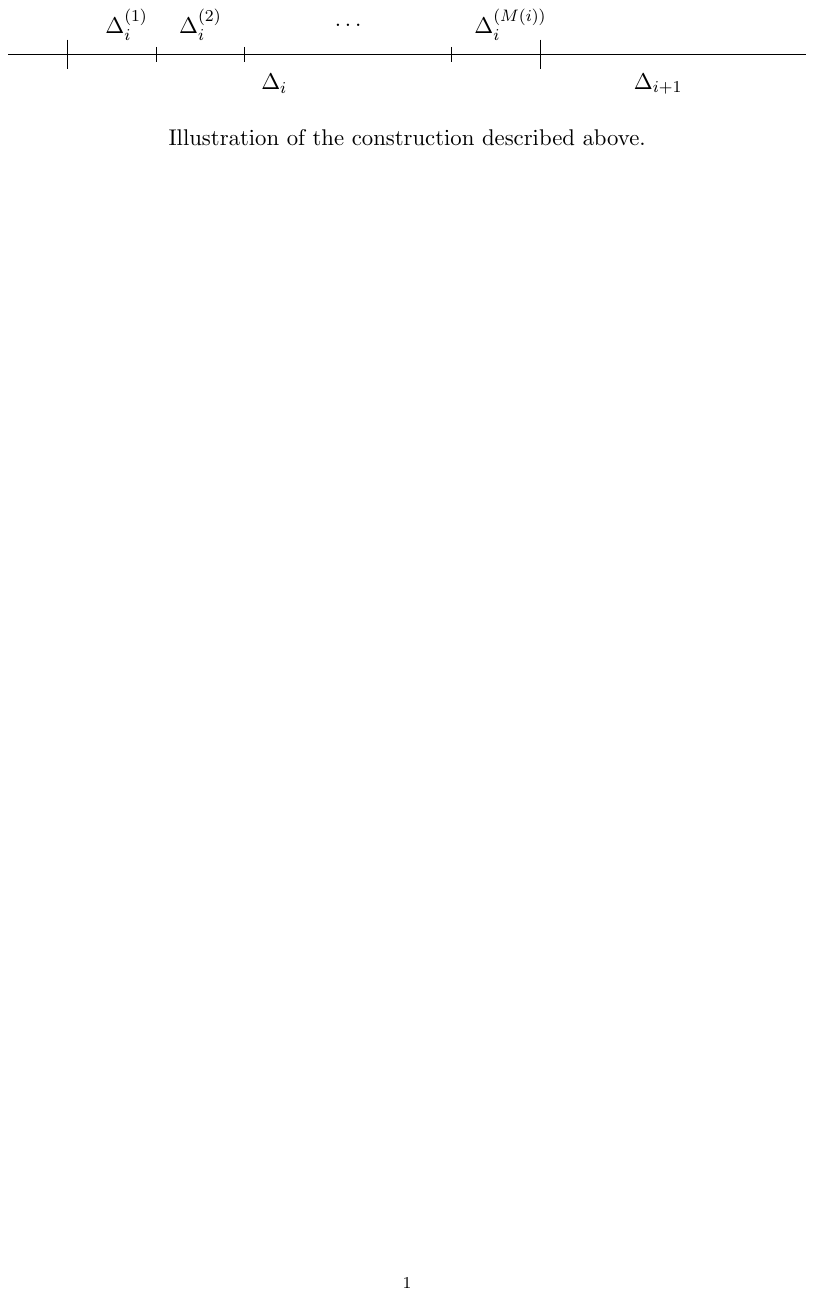}
\end{figure}

\underline{Step 3.} We shall now construct $(n_k)_{k\geq 1}$. For given $i \in\mathbb N$, $m \in \left \{ 1, \ldots , M(i) \right\}$ and for $k \in \Delta_i^{(m)}$, we define
\begin{equation} \label{nk_def}
n_k := 2^{2^{i^4}} \left( 2^k + m \right),
\end{equation}
where the first factor is to be understood as $2^{(2^{(i^4)})}$.

The heuristic behind this construction is the following. First, the elements of our sequence with indices in different blocks $\Delta_{i_1}$ and $\Delta_{i_2}$ are of very different size (because of the dominating prefactor), so that $f(n_k x)$ and $f(n_\ell x)$ for $k \in \Delta_{i_1},~\ell \in \Delta_{i_2},i_1 \neq i_2,$ are ``essentially independent''. Such pairs $k$ and $\ell$ also do not play a relevant role for counting the number of solutions of the Diophantine equations, see Lemma  \ref{lemma_dio} $(ii)$ below. Similarly, pairs $k$ and $\ell$ which are contained in the same block $\Delta_i$, but in different sub-blocks $\Delta_i^{(m_1)}$ resp.\ $\Delta_i^{(m_2)}$, will not play a significant role for counting the number of solutions of the Diophantine equations either, see Lemma \ref{lemma_dio} $(iii)$. What will contribute significantly are only pairs $k,\ell$ from the same sub-block $\Delta_i^{(m)}$, where there will be many solutions of equations such as 
$$
2 n_k - n_\ell = 2^{2^{i^4}} m,
$$
namely whenever $\ell = k+1$ (note that the number $m$ on the right-hand side of the Diophantine equation above is the same as in the superscript of the sub-block $\Delta_i^{(m)}$; different sub-blocks correspond to different Diophantine equations that have ``many'' solutions). In the Erd\H os--Fortet example there are no different blocks whatsoever, so that there are many solutions of the particular equation $2 n_k - n_{k+1} = 1$, which leads to $L(N,2,1,1)$ being as large as $\approx N$. We need $L(N,a,b,c)$ to be smaller in order to satisfy \eqref{dio_LIL_2}, and by grouping $k,\ell$ into approximately $(\log N)^{1-\varepsilon}$ many different blocks, instead of one equation with $\approx N$ solutions, we obtain $(\log N)^{1-\varepsilon}$ different equations with approximately $\frac{N}{(\log N)^{1-\varepsilon}}$ many solutions each, which is in accordance with \eqref{delta_i_m_heuristic}. This explains why the sequence constructed in such a way will satisfy the Diophantine condition \eqref{dio_LIL_2}. It is a different story (and will be shown in Section \ref{sec_LIL}) that the sequence constructed in this very particular way indeed leads to a large value on the right-hand side of the LIL, as claimed by Theorem \ref{th1} (and the presence of the factor $2^{2^{i^4}}$ in the definition of $n_k$ will also only become clear later). We will explain the heuristics behind this part of Theorem \ref{th1} later on, after defining the trigonometric polynomial $f$.\\

Note: Throughout the rest of this section, implied constants are allowed to depend on $\varepsilon$ and $K$ (and consequently also on $d$ and $R$), as well as on $a$ and $b$, but not allowed to depend on anything else. In particular, all implied constants are independent of $c$, $i$, $m$ and $N$. \\

The next lemma provides estimates on the number of solutions of Diophantine equations arising in our setup. As we shall see shortly, from this we can deduce that the sequence $(n_k)_{k\geq 1}$ constructed above indeed satisfies the Diophantine condition \eqref{dio_LIL_2} of Theorem \ref{th1}.  

\begin{lemma} \label{lemma_dio}
For the sequence $(n_k)_{k \geq 1}$ constructed in the paragraph above, we have the following estimates.
\begin{enumerate}[label=(\roman*)]
\item For all $a,b \in \mathbb{N}$ with $a\neq b$, and with $\frac{a}{b} \neq 2^r$ for all $r \in \mathbb{Z}$, we have
$$
\# \{k,\ell \geq 1:~a n_k - b n_\ell = c\} = \mathcal{O}(1),
$$
uniformly in $c \in \mathbb{Z}_{\geq 0}$. 
\item For all $a,b \in \mathbb{N}$ such that $\frac{a}{b}= 2^r$ for some $r \in \mathbb{Z} \backslash \{0\}$, we have
$$
\underbrace{\sum_{i_1=1}^\infty \sum_{i_2=1}^\infty}_{i_1 \neq i_2} \# \left\{k\in \Delta_{i_1},~\ell \in \Delta_{i_2}:~a n_k - b n_\ell = c \right\} = \mathcal{O} (1),
$$
uniformly in $c \in \mathbb{Z}_{\geq 0}$. 
\item For all $a,b \in \mathbb{N}$ such that $\frac{a}{b}= 2^r$ for some $r \in \mathbb{Z} \backslash \{0\}$, and for all $i \in\mathbb N$:
\begin{enumerate}
\item If $c = 2^{2^{i^4}} b m (2^r-1)$ for some $m \in \{1, \dots, M(i)\}$, then
$$
\# \{ k,\ell \in \Delta_i:~ a n_k - b n_\ell = c\} = \frac{R^i}{\lceil i^{1-\varepsilon} \rceil} - r +  \mathcal{O} (i^2).
$$
\item If $c \geq 0$ is not of the form $2^{2^{i^4}} b m(2^r-1)$ for some $m \in \{1, \dots, M(i)\}$, then
$$
\# \{ k,\ell \in \Delta_i:~ a n_k - b n_\ell = c\} =  \mathcal{O}  (i^2).
$$
\end{enumerate}
\end{enumerate}
\end{lemma}

Before proving Lemma \ref{lemma_dio}, we recall a general result about differences of elements of lacunary sequences.

\begin{lemma}[{\cite[p.203]{Zyg}}] \label{zyg_lemma}
Let $(m_k)_{k \geq 1}$ be a positive sequence of integers satisfying \eqref{had} for some $q>1$. Then for all $c \in \mathbb{Z}_{\geq 0}$,
$$
\# \left\{1 \leq k,\ell \leq N,~k \neq \ell:~m_k - m_\ell = c \right\} = \mathcal{O}(1),
$$
where the implied constant depends only on the growth factor $q$.
\end{lemma}

\begin{proof}[Proof of Lemma \ref{lemma_dio}]
Recalling the definition of $L(N,a,b,c)$ in Equation \eqref{L_def}, we assume that $a,b \in \mathbb{N}$ with $a \neq b$, and that $c \in \mathbb{Z}_{\geq 0}$.
\vskip 1mm

$(i)$ Assume $\frac{a}{b}$ is not of the form $2^r$ for any $r \in \mathbb{Z}$. Note that the number $m$ in the definition in line \eqref{nk_def} is much smaller than $2^k$ (for large $k$). Thus, we have (under slight abuse of the limit notation) that, as $k\to\infty$,
\begin{equation} \label{grow_1}
\frac{n_{k+1}}{n_k} \to 2 \qquad \text{on those $k$ and $k+1$ belonging to the same block $\Delta_i$},
\end{equation}
and
\begin{equation} \label{grow_2}
\frac{n_{k+1}}{n_k} \to \infty \qquad \text{on those $k$ and $k+1$ belonging to different blocks $\Delta_i$ and $\Delta_{i+1}$.}
\end{equation}
Since $\frac{a}{b}$ is not a power of 2 by assumption, this shows that the set
\begin{equation} \label{card_of_this}
\# \left\{k,\ell \geq 1:~a n_k - b n_\ell = 0 \right\}
\end{equation}
is finite. To see this, we assume there are infinitely many pairs $k, \ell \in \mathbb{N}$ such that $a n_k - bn_{\ell} = 0$, i.e.,\ $\frac{n_k}{n_\ell} = \frac{a}{b}$. Recall that $a$ and $b$ are assumed to be fixed. Thus, by \eqref{grow_2} there can only be finitely many solutions of this equation for which $k$ and $\ell$ belong to different blocks. Assume  that $\frac{a}{b} > 1$, which means that $\frac{n_k}{n_\ell} = \frac{a}{b}$ is only possible if $k > \ell$ (the case $\frac{a}{b}<1$ can be treated similarly; the case $\frac{a}{b}=1$ is impossible since $\frac{a}{b}$ is not an integer power of 2 by assumption). Since $\frac{a}{b}$ is not an integer power of 2 by assumption, there exists a $\delta>0$ such that $\frac{a}{b} \not\in \bigcup_{j \geq 1} \left[(2-\delta)^j, (2+\delta)^j \right]$. However, on the other hand, by \eqref{grow_1} for all sufficiently large $k$ and $\ell$ with $k > \ell$ which belong to the same block, we have 
\begin{equation*}
    (2- \delta )^{k - \ell} \leq \frac{n_k}{n_{\ell}}  \leq (2 + \delta )^{k - \ell}.
\end{equation*}
Thus, $\frac{n_k}{n_\ell} = \frac{a}{b}$ is possible for only finitely many pairs $(k,\ell)$ which belong to the same block. As noted above, there are also only finitely many solutions $(k,\ell)$ which belong to different blocks. Overall, the cardinality of the set in \eqref{card_of_this} is $\mathcal{O}(1)$. \\

We now form the set-theoretic union
$$
A = \bigcup_{k \geq 1} \{a n_k, b n_k\},
$$
and write the elements of $A$ as a sequence $(m_k)_{k \geq 1}$ (sorted in increasing order). Because of \eqref{grow_1} and \eqref{grow_2}, we have $\liminf_{k \to \infty} \frac{m_{k+1}}{m_k} > 1$ so that the sequence $(m_k)_{k \geq 1}$ is a lacunary sequence with some suitable growth factor (which depends on $a$ and $b$, but these are assumed to be fixed). Thus, by Lemma \ref{zyg_lemma}, we have
$$
\# \left\{k,\ell \geq 1:~a n_k - b n_\ell = c \right\} \leq \# \left\{k,\ell \geq 1,~k \neq \ell:~m_k - m_\ell = c \right\} = \mathcal{O}(1),
$$
where the implied constant in the $\mathcal{O}$-term is independent of $c$; note that the first estimate is trivial since on the right-hand side we have twice as many equations (and thus potential solutions). Summarizing our results, in case (i) of the Lemma we have
\begin{equation} \label{lemma_claim_1}
\# \left\{k,\ell \geq 1:~a n_k - b n_\ell = c \right\} = \mathcal{O} (1),
\end{equation}
uniformly in $c \geq 0$. \\

$(ii)$ Now assume $\frac{a}{b} = 2^r$ for some $r \in \mathbb{Z} \backslash\{0\}$. We first show that there are not many solutions where $k$ and $\ell$ come from different blocks. Assume that $k \in \Delta_{i_1}$ and $\ell \in \Delta_{i_2}$ such that $i_1 < i_2$. Then, whenever $i_2$ is so large that $2^{r+1} \leq 2^{2^{i_2}}$ (which excludes only finitely many values of $i_1$ and $i_2$), using the trivial estimate $i_1 \leq k \leq 2^k$ together with $2^{i_2} + 2^{(i_2-1)^4} < 2^{i_2^4}$, we have
\begin{eqnarray*}
a n_k & \leq & b 2^r 2^{2^{i_1^4}} \left(2^{k} + i_1 \right) \\
& \leq & b 2^{r+1} 2^{2^{i_1^4}} 2^{k} \\
& < & b 2^{2^{i_2}} 2^{2^{(i_2-1)^4}} 2^{\ell}\\
& \leq  & b 2^{2^{i_2^4}} 2^{\ell} \\
& \leq & b n_\ell.
\end{eqnarray*}
Consequently, 
$$
a n_k - b n_\ell = c
$$
is not possible for any non-negative $c$ when $i_2$ is sufficiently large. Assume again that $k \in \Delta_{i_1}$ and $\ell \in \Delta_{i_2}$, but now such that $i_1 > i_2$. Then similar to the previous calculation, assuming that $i_1$ is sufficiently large so that $2^{-r+1} \leq 2^{2^{i_1}}$, and now using that $2^{i_1^4} - 2^{i_1} - 2^{(i_1-1)^4} \geq 2^{i_1^3}$ holds for all $i_1 \in \mathbb{N}$, we have (recall that $\frac{a}{b}=2^r $ and trivially $i_2 \leq \ell \leq 2^{\ell}$)
$$
\frac{a n_k}{b n_\ell} \geq \frac{2^{2^{i_1^4}} 2^k}{2^{-r+1} 2^{2^{i_2^4}} 2^\ell} \geq \frac{2^{2^{i_1^4}}}{2^{2^{i_1}} 2^{2^{(i_1-1)^4}}} \geq 2^{i_1^3}.
$$
Thus, whenever $i_1$ is sufficiently large, then
$$
a n_k \left(1 - 2^{-i_1^3} \right) = a n_k - \frac{a n_k}{2^{i_3^3}} \leq a n_k - b n_\ell  \leq a n_k.
$$
Consequently, for sufficiently large $i_1$ (assuming that ``sufficiently large'' includes the fact that $i_1 \geq 10$), the equality $a n_k - b n_\ell = c$ requires that 
\begin{equation}
\label{Eq_Interval_nk}
n_k \in \left[ \frac{c}{a}, \frac{c}{a}  \left( 1 - 2^{-1000}\right)^{-1} \right].
\end{equation}
We claim that, uniformly in $c \in \mathbb{Z}_{\geq 0}$, there are only finitely many $k \in \mathbb{N}$ such that \eqref{Eq_Interval_nk} holds; more precisely, if $k$ is sufficiently large, then \eqref{Eq_Interval_nk} uniquely determines $k$. Indeed, by \eqref{grow_1} and \eqref{grow_2} we have $\frac{n_{k}}{n_{k-1}}\geq \frac{3}{2}$ for all sufficiently large $k$. Thus whenever we have \eqref{Eq_Interval_nk}, then 
$$
n_{k-1} \leq \frac{2}{3} \frac{c}{a}  \left( 1 - 2^{-1000}\right)^{-1} \not\in  \left[ \frac{c}{a}, \frac{c}{a}  \left( 1 - 2^{-1000}\right)^{-1} \right]
$$
as well as
$$
n_{k+1} \geq \frac{3}{2} \frac{c}{a} \not\in \left[ \frac{c}{a}, \frac{c}{a}  \left( 1 - 2^{-1000}\right)^{-1} \right], 
$$
with the possible exception of finitely many indices. Thus, overall we have shown that in the case $\frac{a}{b} = 2^r$ for some $r \in \mathbb{Z} \backslash \{0\}$, we have
\begin{equation} \label{lemma_claim_2}
\underbrace{\sum_{i_1=1}^\infty \sum_{i_2=1}^\infty}_{i_1 \neq i_2} \# \left\{k\in \Delta_{i_1},~\ell \in \Delta_{i_2}:~a n_k - b n_\ell = c \right\} = \mathcal{O} (1),
\end{equation}
with an implied constant that is independent of $c \geq 0$. This settles case (ii) of the lemma.\\

$(iii)$ Now we are in the situation where $k, \ell$ are contained in the same block $  \Delta_i$ for some $i \in\mathbb N$. We establish a few general estimates which we shall use in the proof of both $(iii)$ $ a)$ and $(iii) $ $b)$. 

Recall that we are in the case where there exists an $r \in \mathbb{Z} \backslash \{ 0\}$ with $\frac{a}{b}= 2^r$ and let $k , \ell \in \Delta_i$ with $k + r > \ell$. Then, for suitable $m_1,m_2 \in \{ 1, \ldots, M(i) \}$, we have
\begin{eqnarray*}
a n_k - b n_\ell & = & 2^r b n_k - b n_\ell \\
& = & 2^{2^{i^4}} b (2^r 2^k +m_1  - 2^\ell - m_2) \\
& = & 2^{2^{i^4}} b (2^{k+r} +m_1  - 2^\ell - m_2).
\end{eqnarray*}
Thus, $a n_k - b n_\ell = c$ can only hold when 
$$2^{k+r}- 2^\ell = \frac{2^{2^{i^4}} b (m_2 - m_1) + c}{2^{2^{i^4}} b },
$$
and for given $b,c,r,m_1,m_2$ (whence the right-hand side of this equation is fixed) there is at most one pair $(k,\ell)$ with $k+r > \ell$ for which this can hold (this is essentially the uniqueness of the dyadic representation of integers). There are $\mathcal{O}(M(i)^2) = \mathcal{O} (i^{2})$ many possible values for $m_1, m_2$, so the total number of solutions $(k,\ell)$ of $a n_k - b n_\ell = c$ with $k + r > \ell$ and $k,\ell \in \Delta_i$ is at most $\mathcal{O}(i^{2})$, uniformly in $c \geq 0$. The same argument can be applied to the case $k + r < \ell$. Thus, we have established
\begin{equation}
    \label{Eq_Sol_Delta_i}
    \sup_{c \in \mathbb{Z}_{\geq 0} } \# \left \{ k, \ell \in \Delta_i : k+ r \neq \ell,  a n_k - b n_{\ell} = c \right\} =  \mathcal{O}(i^2).
\end{equation}
The number of pairs of indices $k$ and $\ell$ with $k + r = \ell$, which are contained in different sub-blocks $k \in \Delta_i^{(m_1)}$ and $\ell \in \Delta_i^{(m_2)}$ for $m_1 \neq m_2$, is $\mathcal{O} (M(i)) = \mathcal{O} (i)$. Thus, by \eqref{Eq_Sol_Delta_i} we have
\begin{align}
\notag
   \# \left \{ k, \ell \in \Delta_i :  a n_k - b n_{\ell} = c \right \}
   &=   \sum_{m'=1}^{M(i)} \#  \{ k, \ell \in \Delta_i^{(m')} : k+ r = \ell , a n_k - b n_{\ell} = c \} +  \mathcal{O}(i) + \mathcal{O}(i^2)\\
   \label{Eq_Sol_Delta_i_2}
   &=  \sum_{m'=1}^{M(i)} \#  \{ k, \ell \in \Delta_i^{(m')} : k+ r = \ell , a n_k - b n_{\ell} = c \} + \mathcal{O}(i^2).
\end{align}
In the following, we only consider the case where $r > 0$ (the case $r<0$ can be treated in an analogous way). We need to count the number of $k$ and $\ell$ such that $k+r=\ell$ and $k,\ell \in \Delta_i^{(m')}$ for some $m' \in \left \{1, \ldots, M(i) \right\}$, where we have
\begin{equation}
\label{Eq_Sol_Delta_i_m}
\begin{split}
a n_k - b n_\ell 
& = 2^r b n_k - b n_\ell \\
&  =  b \left( 2^r n_k - n_\ell \right) \\
& = 2^{2^{i^4}} b (2^r (2^k + m') - (2^\ell + m')) \\
& =  2^{2^{i^4}} b (\underbrace{2^{r+k} - 2^\ell}_{=0} + 2^r m' - m')) \\
& =  2^{2^{i^4}} b m' (2^r-1).
\end{split}
\end{equation}
Now we prove $(iii)$ $a)$, where we assumed that $c$ is of the special form $c= 2^{2^{i^4}} bm (2^r-1)$ for some $m \in \{1, \ldots, M(i) \}$. By \eqref{Eq_Sol_Delta_i_m}, the equation $a n_k - b n_{\ell} = c$ is satisfied if and only if $m=m'$ and thus 
\begin{equation}
\label{Eq_Major_Contr}
\begin{split}
\sum_{m'=1}^{M(i)} \# \left \{ k, \ell \in \Delta_i^{(m')} : k+ r = \ell , a n_k - b n_{\ell} = c \right \} 
&= 
\# \{k,\ell \in \Delta_i^{(m)}:~k + r= \ell\}  \\
& = \# \Delta_i^{(m)} - r.
\end{split}
\end{equation}
Combining \eqref{mi_card}, \eqref{Eq_Sol_Delta_i_2} and \eqref{Eq_Major_Contr}, we get in case $(iii)$ $a)$
\begin{equation*}
    \# \left \{ k, \ell \in \Delta_i : a n_k - b n_{\ell} = c \right \} = \frac{R^i}{\lceil i^{1- \varepsilon} \rceil } - r + \mathcal{O}(i^2), 
\end{equation*}
as desired.\\

Now assume that we are in case $(iii)$ $b)$, i.e., $c$ is not of the form $ 2^{2^{i^4}} b m (2^r -1)$ for any $m \in \left \{ 1, \ldots , M(i) \right \}$. Then \eqref{Eq_Sol_Delta_i_m} yields
\begin{equation*}
  \sum_{m'=1}^{M(i)} \# \left \{ k, \ell \in \Delta_i^{(m')} : k+ r = \ell , a n_k - b n_{\ell} = c \right \} = 0  
\end{equation*}
and from \eqref{Eq_Sol_Delta_i_2}, we obtain
\begin{equation*}
\# \left \{ k, \ell \in \Delta_i :  a n_k - b n_{\ell} = c \right \}=  \mathcal{O}(i^2),    
\end{equation*}
as claimed.
\end{proof}

We can now prove that our gap sequence $(n_k)_{k\geq 1}$ satisfies the desired Diophantine condition.

\begin{cor} \label{cor_dio}
The sequence $(n_k)_{k \geq 1}$ constructed in this section satisfies the Diophantine condition \eqref{dio_LIL_2} of Theorem \ref{th1}.
\end{cor}
\begin{proof}
Let $a,b\in\mathbb N$ be fixed. Let $N\in\mathbb N$ be given, and let $I \in\mathbb N$ be such that $N \in \Delta_I$. We observe the following (note that $ \left( \frac{i+1}{i} \right)^{1- \varepsilon} \leq 2$ for all $i \in \mathbb{N}$)
\begin{align*}
    \sum_{i=1}^I \frac{R^i}{i^{1-\varepsilon}} & =\frac{R^I}{I^{1-\varepsilon}} +  \sum_{i=1}^{I-1} \frac{R^i}{i^{1-\varepsilon}} \\
    & \leq \frac{R^I}{I^{1-\varepsilon}} + \frac{2}{R} \sum_{i=1}^{I} \frac{R^i}{i^{1-\varepsilon}}.
\end{align*}
Rearranging the terms and using the fact that there exists a constant $c > 0$ (depending on $R)$ such that $ c \log N  \leq I $ and $R^I \leq R N $ yields 
\begin{equation*}
\sum_{i=1}^I \frac{R^i}{i^{1-\varepsilon}}  \leq \left( 1 -\frac{2}{R} \right)^{-1} \frac{R^I}{I^{1-\varepsilon}} = \mathcal{O} \left( \frac{ N}{ (\log N) ^{1- \varepsilon}} \right).     
\end{equation*}
Combining all the worst-case estimates from Lemma \ref{lemma_dio}, the previous calculation shows that
$$
\# \{k,\ell \leq N:~a n_k - b n_\ell = c\} = \mathcal{O}\left( \sum_{i=1}^I \frac{R^i}{i^{1-\varepsilon}} \right) = \mathcal{O}\left( \frac{N}{(\log N)^{1-\varepsilon}} \right), 
$$
uniformly in $c \in \mathbb{Z}_{\geq 0}$, as claimed.
\end{proof}

\section{Further ingredients -- Gaposhkin's Berry-Esseen result}

We will need a Berry--Esseen type quantitative central limit theorem for the particular lacunary trigonometric sum $\sum_{k=1}^N \cos( 2\pi 2^k x)$. As in the introduction, $\lambda$ denotes Lebesgue measure and $\Phi$ denotes the standard normal distribution function.
\begin{lemma}[Gaposhkin \cite{gapo}] \label{lemma_gaposhkin}
Let $\lambda_1, \dots, \lambda_N$ be non-negative real numbers such that 
$$
\sum_{k=1}^N \lambda_k^2 = 1. 
$$
Set $\Lambda_N = \max_{1 \leq k \leq N} \lambda_k$. Then
$$
\sup_{t \in \mathbb{R}}\left| \lambda\left( \left\{ x \in (0,1):~ \sqrt{2} \sum_{k=1}^N \lambda_k \cos( 2\pi 2^k x) < t \right\} \right) - \Phi(t) \right| = \mathcal{O} \left(\Lambda_N^{1/4} \right), 
$$
where the implied constant is absolute.
\end{lemma}

\section{The law of the iterated logarithm} \label{sec_LIL}
Let $\varepsilon \in (0,1)$ and $K \in (0, \infty)$. We now define our trigonometric polynomial
$$
f(x) = \sum_{j=0}^{d-1} \cos (2 \pi 2^j x),
$$
where $d$ satisfies \eqref{size_d}, i.e., $ \frac{d \sqrt{\varepsilon} }{4} -2 \geq \frac{K \sqrt{d}}{\sqrt{2}} $. We will prove that for this trigonometric polynomial $f$, and for the gap sequence $(n_k)_{k \geq 1}$ we constructed in Section \ref{sec_construction}, the conclusion of Theorem \ref{th1} is indeed satisfied.\\

Clearly, since our cosine functions are uncorrelated, 
\begin{equation} \label{2-norm}
\| f\|_2 = \frac{\sqrt{d}}{\sqrt{2}}.
\end{equation}
Applying the Erd\H os--G\'al law of the iterated logarithm (see Equation \eqref{erd-gal}) together with the triangle inequality gives
\begin{equation} \label{erd-gal-2}
\limsup_{N \to \infty} \frac{\left| \sum_{k=1}^N f(n_k x) \right|}{\sqrt{2 N \log \log N}} \leq \frac{d}{\sqrt{2}} \qquad \text{a.e.}
\end{equation}
Recall that $\# \Delta_i = R^i$. We will show that
\begin{equation} \label{to_show}
\limsup_{i \to \infty} \frac{\left| \sum_{k \in \Delta_i} f(n_k x) \right|}{\sqrt{2 R^i \log \log R^i}} \geq \frac{d \sqrt{\varepsilon}}{2} - 2 \qquad \text{a.e.}
\end{equation}
Assuming this to be true, then with the notation $N(i) := \frac{R^{i+1} - R}{R-1}$ and upon noting that $\{1, \dots, N(i)\} = \Delta_1 \cup \dots \cup \Delta_i = \{ 1, \dots, N(i-1) \} \cup \Delta_i$, for almost all $x \in [0,1]$, we will obtain
\begin{eqnarray*}
 & & \limsup_{i \to \infty} \frac{\left| \sum_{k = 1}^{N(i)} f(n_k x) \right|}{\sqrt{2 N(i) \log \log N(i)}} \\
 & \geq & \limsup_{i \to \infty} \frac{\left| \sum_{k \in \Delta_i} f(n_k x) \right|}{\sqrt{2 N(i) \log \log N(i)}} - \limsup_{i \to \infty} \frac{\left| \sum_{k =1}^{N(i-1)} f(n_k x) \right|}{\sqrt{2 N(i) \log \log N(i)}} \\
& \geq & \underbrace{\limsup_{i \to \infty} \frac{\left| \sum_{k \in \Delta_i} f(n_k x) \right|}{\sqrt{2 R^i \log \log R^i}}}_{\geq \frac{d \sqrt{\varepsilon}}{2}-2 \text{ by \eqref{to_show}}} 
 \underbrace{ \frac{ \sqrt{2 R^{i} \log \log R^{i}}}{\sqrt{2 N(i) \log \log N(i)}}}_{\longrightarrow 1 \text{ as } i \rightarrow \infty} \\
& & - \underbrace{\limsup_{i \to \infty} \frac{\left| \sum_{k =1}^{N(i-1)} f(n_k x) \right|}{\sqrt{2 N(i-1) \log \log N(i-1)}}}_{\leq \frac{d}{\sqrt{2}} \text{ by } \eqref{erd-gal-2}} \underbrace{\lim_{i \to \infty} \frac{\sqrt{2 N(i-1) \log \log N(i-1)}}{\sqrt{2 N(i) \log \log N(i)}}}_{= \frac{1}{\sqrt{R}} \leq \frac{\sqrt{\varepsilon}}{\sqrt{8}} \text{ by \eqref{size_R}}} \\
& \geq & \frac{d \sqrt{\varepsilon}}{4}-2.
\end{eqnarray*}
By \eqref{size_d} and \eqref{2-norm}, we have $\frac{d \sqrt{\varepsilon}}{4}-2 > K \|f\|_2$. Thus, it remains to establish \eqref{to_show}. We note in passing that this chain of calculations actually gives a quantitative form of Theorem \ref{th1}, where the dependence between the factor $K$, the size of $\varepsilon$ from the savings in the Diophantine condition, and the degree $d$ of the trigonometric polynomial are brought into relation. The conclusion of Theorem \ref{th1} is non-trivial when the right-hand side exceeds $\|f\|_2$. One can check that for $K=1$, the inequality \eqref{size_d} holds for $d = 21 \varepsilon^{-1}$, so that for given $\varepsilon$ from the Diophantine condition, we can construct a counterexample to the LIL in its ``truly independent'' form by considering a trigonometric polynomial of degree $\lceil 21 \varepsilon^{-1} \rceil$. As $\varepsilon$ in the Diophantine condition approaches 0, we need to increase the degree of the trigonometric polynomial to get a result which doesn't match with the ``truly independent'' form of the LIL. The main result of this paper is that the Diophantine condition \eqref{dio_LIL} is essentially  optimal when trigonometric polynomials of arbitrary degree are considered; however, we do emphatically \emph{not} claim that this condition is also optimal in the case of trigonometric polynomials whose degree is bounded. For degree 1 (the Erd\H os--G\'al case), no Diophantine condition at all is necessary. For trigonometric polynomials of degree 2, we believe that a condition of order roughly $L(N,a,b,c) = \mathcal{O} \left( \frac{N}{(\log N)^{1/2}} \right)$ should be optimal, and for degree at most $d$, a condition of the form $L(N,a,b,c) = \mathcal{O} \left( \frac{N}{(\log N)^{c(d)}} \right)$ should be optimal, with some suitable constants $c(d)$ for which $c(d) \nearrow 1$ as $d \to \infty$. The relation between $d$ and $\varepsilon$ was not in our main focus when writing this paper, but this seems to be an interesting topic for further research.\\

It remains to prove \eqref{to_show}. For $i \in\mathbb N$, let $\mathcal{F}_i$ be the sigma-field generated by the collection of intervals
$$
\left[ \frac{a-1}{2^{2^{(i+1)^4}}},\frac{a}{2^{2^{(i+1)^4}}} \right), \qquad a = 1, 2, \dots, 2^{2^{(i+1)^4}}.
$$
We set 
$$
  Y_i(x) : = \sum_{k \in \Delta_i} f(n_k x),\quad x \in [0,1],
$$   
as well as
$$
Z_i := \mathbb{E} \left( Y_i \Big| \mathcal{F}_i \right), \qquad i \geq 1,
$$
which, as we shall see in a moment, is a good approximation of the random variable $Y_i$; in probabilistic parlance, the system $(\mathcal{F}_i)_{i \geq 1}$ forms a filtration of the unit interval. Using that
$$
\|f ' \|_\infty \leq \sum_{j=0}^{d-1} 2 \pi 2^j \leq 2 \pi 2^d,
$$
and setting $I_a^i := \left [\frac{a-1}{2^{2^{(i+1)^{4}}}}, \frac{a}{2^{2^{(i+1)^{4}}}} \right)$ for $a \in \left \{ 1, \ldots, 2^{2^{(i+1)^4}} \right\} =: S_i $, where $i \in\mathbb N$, we see that by using the standard estimate $| f(x) -f(y) | \leq \| f' \|_{\infty} | x- y|$,
\begin{align*}
\|Z_i - Y_i \|_\infty &= \sup_{x \in [0,1]} \left | Z_i(x) - Y_i(x) \right | \\
& =\sup_{x \in [0,1]} \left | \sum_{a \in S_i} \sum_{k \in \Delta_i}\left( 2^{2^{(i+1)^4}}   \int_{I_a^i} f(n_kt) dt - f(n_k x) \right) \mathbbm{1}_{I_a^i}(x)  \right | \\
& \leq \max_{a \in S_i} \sup_{x \in I_a^i} \sum_{k \in \Delta_i} \left |  2^{2^{(i+1)^4}}\int_{I_a^i}  \left(f(n_k t)  -  f(n_k x) \right) dt \right | \\
& \leq  \max_{a \in S_i } \sup_{x \in I_a^i} \sum_{k \in \Delta_i} \frac{ ||f'||_{\infty } n_k}{2^{2^{(i+1)^4}}} \\
& = \sum_{k \in \Delta_i} \frac{2 \pi 2^d n_k}{2^{2^{(i+1)^4}}} \\
& \leq \frac{R^i 2 \pi 2^d 2^{2^{i^4}}\left(2^{R^{i+1}} + i \right)}{2^{2^{(i+1)^4}}}, 
\end{align*}
which goes rapidly to zero as $i \to \infty$ (recall that $R$ and $d$ are fixed). Thus, we have 
$$
\limsup_{i \to \infty} \frac{\left| Y_i \right|}{\sqrt{2 R^i \log \log R^i}} \geq \frac{d \sqrt{\varepsilon}}{2} - 2 \qquad \textup{a.e.},
$$
(which is just another way of writing \eqref{to_show}) if and only if
\begin{equation} \label{to_show_2}
\limsup_{i \to \infty} \frac{\left| Z_i \right|}{\sqrt{2 R^i \log \log R^i}} \geq \frac{d \sqrt{\varepsilon}}{2} - 2 \qquad \textup{a.e.},
\end{equation}
and our aim thus becomes to establish \eqref{to_show_2}. For this purpose, we define the sets
$$
A_i := \left\{ x \in [0,1]: ~\left| Z_i \right| \geq \left(\frac{d \sqrt{\varepsilon}}{2}-2 \right) \sqrt{2 R^i \log \log R^i} - 2 d^2 i - 3 \right\}, \qquad i \in \mathbb N,
$$
where the terms which are subtracted on the right-hand side will allow us to incorporate errors which will appear later on in the computations. By construction, $A_i$ is $\mathcal{F}_i$-measurable for all $i \geq 1$. We claim that $A_i$ is independent of $\mathcal{F}_{i-1}$ (and hence also independent of $\mathcal{F}_{i-2}, \mathcal{F}_{i-3}, \ldots $). This follows from the fact that all numbers $n_k$ with $k \in \Delta_i$ are integer multiples of $2^{2^{i^4}}$ and hence the functions $Y_i$ and $Z_i$ are periodic with period-length $\frac{1}{2^{2^{i^4}}}$. From that we infer, again writing $ I_a^{i-1}= \left [\frac{a-1}{2^{2^{i^{4}}}}, \frac{a}{2^{2^{i^{4}}}} \right)$ for $a \in S_{i-1} = \{1, \ldots, 2^{2^{i^4}} \}$, that $\lambda (A_i \cap I_a^{i-1})$ has the same value for all $a \in S_{i-1}$. Thus, for all $x \in [0,1]$, we obtain
\begin{align*}
 \mathbb{E} \left[ \mathbbm{1}_{A_i} | \mathcal{F}_{i-1} \right] (x)
 &= \sum_{a \in S_{i-1}} 2^{2^{i^4}} \lambda( A_i \cap I_a^{i-1} ) \mathbbm{1}_{I_a^{i-1}}(x) \\
 &= 2^{2^{i^4}} \lambda (A_i \cap I_1^{i-1}) \\
 & = \sum_{a \in S_{i-1}} \lambda (A_i \cap I_a^{i-1}) \\
 & = \lambda(A_i),
\end{align*}
which proves our claim. Hence, the sets $A_1, A_2, A_3, \dots$ are stochastically independent -- this was the purpose of the factor $2^{2^{i^4}}$ in the definition of $(n_k)_{k\geq 1}$ in \eqref{nk_def} for all $k$ from the same block $\Delta_i$, and for switching from $Y_i$ to the discretized approximations $Z_i$.\footnote{Our aim is to prove that for all $i$ the lacunary sum $\sum_{k \in \Delta_i} f(n_k x)$ is large for a sufficiently large set of values of $x$, which yields the desired LIL by an application of the divergence Borel--Cantelli lemma. Without the factor $2^{2^{i^4}}$ in the definition of $n_k$ for all $k$ from the same block $\Delta_i$ we would lack the necessary stochastic independence of these set, which is a necessary prerequisite for an application of the divergence Borel--Cantelli lemma. More specifically, in what follows we will decompose $\sum_{k \in \Delta_i} f(n_k x)$ into the product of a ``local variance function'', multiplied with a pure trigonometric sum (whose distribution is close to Gaussian), and we have to make sure that the local variance function is not large at the same locations of $x$, for different values of $i$.}\\

It remains to establish that
\begin{equation} \label{measure_diverges}
\sum_{i=1}^\infty \lambda(A_i) = +\infty.
\end{equation}
Then, by an application of the second Borel-Cantelli lemma (using the independence of the sets $A_1, A_2, \dots$), we can conclude that almost all $x\in[0,1]$ are contained in infinitely many sets $A_i$, which implies \eqref{to_show_2} and \eqref{to_show}.\\

As observed before, we have $\|Y_i - Z_i \|_\infty \leq 1$ for all sufficiently large $i \in\mathbb N$. Thus, by the triangle inequality
\begin{equation} \label{A_sup_set}
A_i \supseteq \left\{ x \in [0,1]: ~\left| Y_i \right| \geq \left(\frac{d \sqrt{\varepsilon}}{2}-2 \right) \sqrt{2 R^i \log \log R^i} - 2 d^2 i - 2 \right\}
\end{equation}
for all sufficiently large $i \in\mathbb N$. As noted, $Y_i$ is periodic with period $2^{2^{i^4}}$. We now define a new sequence $(\nu_k)_{k \geq 1}$ via
$$
\nu_k := \frac{n_k}{2^{2^{i^4}}}, \qquad k \in \Delta_i, \quad i \in\mathbb N,
$$
or, equivalently, via
$$
\nu_k = 2^k +m, \qquad k \in \Delta_i^{(m)}, \quad m \in \{1, \dots, M(i)\},~i\in\mathbb N .
$$
Then $Y_i$ has the same distribution as $\sum_{k \in \Delta_i} f(\nu_k x)$ so that
\begin{eqnarray}
& & \lambda \left( \left\{ x \in [0,1]: ~\left| Y_i \right| \geq \left(\frac{d \sqrt{\varepsilon}}{2}-2 \right)\sqrt{2 R^i \log \log R^i} - 2 d^2 i - 2 \right\} \right) \nonumber \\
& = & \lambda \left(\left\{ x \in [0,1]: ~\left| \sum_{k \in \Delta_i} f(\nu_k x) \right| \geq \left(\frac{d \sqrt{\varepsilon}}{2}-2 \right) \sqrt{2 R^i \log \log R^i} - 2 d^2 i - 2 \right\} \right).  \label{A_set_nu}
\end{eqnarray}
Now we will relate the particular choice of our function $f$ with our particular construction of the sequence $(n_k)_{k \geq 1}$, somewhat in the spirit of the Erd\H os--Fortet example. The key point of the Erd\H os--Fortet example is that in the sum $\sum \left( \cos(2 \pi (2^k-1) x) + \cos(4 \pi (2^k - 1) x) \right)$, the term $\cos(4 \pi (2^k - 1) x) = \cos(2 \pi (2^{k+1} - 2) x)$ (from frequency $4 \pi$ and index $k$) and the term $\cos (2 \pi (2^{k+1} - 1) x)$ (from frequency $2\pi$ and index $k+1$) can be combined, so that by the standard trigonometric identity $\cos ( \alpha) + \cos(  \beta) = 2 \cos \left( \frac{\alpha + \beta }{2} \right) \cos \left( \frac{\alpha - \beta }{2} \right) $, we have
\begin{eqnarray}
\sum_{k=1}^N \left( \cos(2 \pi (2^k-1) x) + \cos(4 \pi (2^k - 1) x) \right) 
& = &\sum_{k=1}^N  \left( \cos(2 \pi (2^k-1) x) + \cos(2 \pi (2^{k} - 2) x) \right) \nonumber  \\
&  & \qquad     + \cos(2 \pi (2^{N+1} -2) x) -1 \nonumber \\
& \approx & \sum_{k=1}^N \left( \cos(2 \pi (2^k-1) x) + \cos(2 \pi (2^{k} - 2) x) \right) \nonumber \\
& =  & 2 \cos(\pi x) \sum_{k=1}^N  \cos(2 \pi (2^k-3/2) x), \label{erd-fort-2}
\end{eqnarray}
i.e.,\ the generalized lacunary sum essentially decomposes into the product of the fixed function $2 \cos(\pi x)$ and a purely trigonometric lacunary sum (this explains why the factor $2 \cos (\pi x)$ appears on the right-hand side of  \eqref{erd-fort}). Our sum $\sum_{k \in \Delta_i} f(\nu_k x)$ will split in a somewhat similar way into a (slowly fluctuating) function $g_i(x)$, multiplied with a purely trigonometric lacunary sum (to which we can apply Gaposhkin's quantitative CLT of Lemma \ref{lemma_gaposhkin}). However, while in Erd\H os--Fortet's construction the contribution of only two subsequent summation indices can be combined (leading to a factor $2 \cos(\pi x)$ which is bounded by 2), in our construction the contribution of $d$ subsequent summation indices can be combined, leading to a function $g_i(x)$ which becomes as large as $d$ for some values of $x$. We will continue to comment on the heuristics after some further steps of calculations.\\

For $i \geq 1$ and $m \in \{1, \dots, M(i) \}$, we have
\begin{eqnarray}
\sum_{k \in \Delta_i^{(m)}} f(\nu_k x) & = & \sum_{k \in \Delta_i^{(m)}} \sum_{j=0}^{d-1} \cos (2 \pi 2^j (2^k + m) x) \nonumber\\
& = & \sum_{k \in \Delta_i^{(m)}} \sum_{j=0}^{d-1} \cos (2 \pi (2^{k+j} + 2^j m) x). \label{here_appear}
\end{eqnarray}
The sequence $(n_k)_{k \geq 1}$ (and so $(\nu_k)_{k \geq 1}$) and the trigonometric polynomial $f$ were constructed in such a way that in the representation \eqref{here_appear} there are many terms that can be combined. More precisely, when understanding $k+j= \ell$ as a new summation index in \eqref{here_appear}, then there are $d$ many different pairs $(k,j)$ for which $k+j$ adds up to $\ell$. Thus, the sum in \eqref{here_appear} can be rewritten as
\begin{equation} \label{rewritten_as}
\sum_{k \in \Delta_i^{(m)}} \sum_{j=0}^{d-1} \cos (2 \pi (2^{k+j} + 2^j m) x) = \sum_{\ell \in \Delta_i^{(m)}} \sum_{j=0}^{d-1} \cos (2 \pi (2^\ell + 2^j m) x) + E_{i,m}(x),
\end{equation}
where $E_{i,m}(x)$ is an error term consisting of sums of cosines which comes from a) the $d$ smallest indices $\ell \in \Delta_{i}^{(m)}$, for which there do not exist $d$ many pairs $(k,j)$ with $k \in \Delta_i^{(m)}$ and $j \in \{0,\dots,d-1\}$ such that $k+j=\ell$, and b) from the $d$ largest indices $k$ in $\Delta_i^{(m)}$ for which $k+j$ exceeds all elements $\ell \in \Delta_i^{(m)}$.\footnote{Let $\left \{ k_1, \ldots, k_{\# \Delta_i^{(m)}} \right \}$ denote the elements of $\Delta_i^{(m)}$ in increasing order. Then $k_1$ has exactly one representation of the form $\ell +j$ for $\ell \in \Delta_i^{(m)}$ and $j \in \{0, \ldots, d-1 \}$, namely $k_1 = k_1 +0$. $k_2$ has two such representations, namely $k_2 = k_2 + 0$ and $k_2 = k_1 +1$. So, the first $d-1$ elements of $\Delta_i^{(m)}$ have less than $d$ possible representations of the form $\ell +j$. This gives us precisely $\frac{d(d-1)}{2}$ many cosine functions which form one part of $E_{i,m}$. On the other hand, if $\ell $ is one of the $d-1$ largest elements in $\Delta_i^{(m)}$, then we have strictly less than $d$ choices for $j \in \{ 0, \ldots, d-1\}$ such that $\ell + j \in \Delta_i^{(m)}$. This leads to another $\frac{d(d-1)}{2}$ summands in $E_{i,m}$. In total $E_{i,m}$ consists of $d (d-1) \leq d^2$ many cosine functions.} Hence, $E_{i,m}$ is a sum of at most $d^2$ many cosine functions, and thus
\begin{equation*}
    \left | \left | E_{i,m} \right| \right |_{\infty} \leq d^2 
\end{equation*}
as well as (recall $ M(i) = \lceil  i^{1- \varepsilon} \rceil \leq i $)
\begin{equation} \label{eim_error}
\left\|\sum_{m=1}^{M(i)} E_{i,m} \right\|_\infty \leq  d^2 i.
\end{equation} \\

Let $w \in \mathbb{N}$. Then, applying the trigonometric identity $\cos(x+y) = \cos x \cos y - \sin x \sin y$, we have
\begin{eqnarray*}
\sum_{j=0}^{d-1} \cos (2 \pi (w + 2^j m) x) & = &  \cos (2 \pi w x) \sum_{j=0}^{d-1} \cos (2 \pi 2^j m x) - \sin (2 \pi w x) \sum_{j=0}^{d-1} \sin (2 \pi 2^j m x).
\end{eqnarray*}
Applying this to the sum in \eqref{rewritten_as}, together with \eqref{here_appear} and the trigonometric identity $\cos(2 \pi y) = 1 - 2 (\sin(\pi y))^2,~y \in \mathbb{R}$, after summing over $m = 1, \dots, M(i),$ we arrive at
\begin{eqnarray}
& & \sum_{k \in \Delta_i} f(\nu_k x)  \nonumber\\
& = & \sum_{m=1}^{M(i)} \left(\sum_{j=0}^{d-1} \cos (2 \pi 2^j m x)  \sum_{k \in \Delta_i^{(m)}} \cos (2 \pi 2^k x) \right. \nonumber\\
& & \quad \left. - \sum_{j=0}^{d-1} \sin (2 \pi 2^j m x) \sum_{k \in \Delta_i^{(m)}}  \sin (2 \pi 2^k x)  + E_{i,m}(x) \right) \nonumber\\
& = &  d \sum_{k \in \Delta_i} \cos (2 \pi 2^k x) \nonumber \\
& & \quad -  \sum_{m=1}^{M(i)} \sum_{j=0}^{d-1} (\sin (\pi 2^j m x))^2 \sum_{k \in \Delta_i^{(m)}} \cos (2 \pi 2^k x)  \label{many_terms_2}\\
& & \quad - \sum_{m=1}^{M(i)} \sum_{j=0}^{d-1} \sin (2 \pi 2^j m x)   \sum_{k \in \Delta_i^{(m)}}  \sin (2 \pi 2^k x)  \label{many_terms_3} \\
& &  \quad + \sum_{m=1}^{M(i)}  E_{i,m}(x). \label{many_terms_4}
\end{eqnarray}
We comment one last time on the heuristics behind our construction of the function $f$ and the sequence $(n_k)_{k \geq 1}$. They were carefully adapted to each other so that (after removing the extra periodicity by changing from $n_k$ to $\nu_k$), we can essentially write 
$$
\sum_{k \in \Delta_i} f(\nu_k x) = g_i(x) \sum_{k \in \Delta_i} \cos (2 \pi 2^k x) + (\text{errors}),
$$
where $g_i(x)$ is a function which becomes as large as $d$ when $x$ is small (we ignore here the fact that in the equations above, one lacunary sum is a sine sum, not a cosine sum). Heuristically (ignoring that $g_i(x) $ also depends on $m$), we essentially have $g_i(x) = d - g_{i}^{(1)} (x) - g_i^{(2)} (x)$, where $g_i^{(1)}$ and $g_i^{(2)}$ are the slowly fluctuating functions in the double sums in lines \eqref{many_terms_2} and \eqref{many_terms_3}, respectively. We can ensure that $g_i^{(1)}$ and $g_i^{(2)}$ are small when $x$ is smaller than $1/M(i) \approx i^{-1+\varepsilon}$ by some factor. The sum $R^{-i/2} \sum_{k \in \Delta_i} \cos (2 \pi 2^k x)$ is a classical normalized lacunary sum and behaves like a Gaussian $\mathcal{N} (0,1/2)$ random variable (see, e.g., \cite[Theorem 1]{kac}). Thus, for $x$ near $0$, the sum $R^{-i/2}\sum_{k \in \Delta_i} f(\nu_k x)$ essentially behaves like $R^{-i/2} d \sum_{k \in \Delta_i} \cos (2 \pi 2^k x)$ and thus like a $\mathcal{N} (0,d^2/2)$ random variable (locally for $x$ near 0 we have gained a factor $d$ for the variance in comparison with $\|f\|_2^2$, this is the key point!), and we can factorize
\begin{eqnarray*}
& & \text{Prob} \left(x\in [0,1] \,:\, \left| \sum_{k \in \Delta_i} f(\nu_k x) \right| \text{ is ``large''} \right) \\ & \geq & \textup{Prob} \big(x \text{ is ``close enough'' to 0 so that $g_i(x)\approx d$} \big)  \times  \textup{Prob} \big(\text{a $\mathcal{N} (0,d^2/2)$ r.v.\ is ``large''}\big).
\end{eqnarray*}
The size of the set of values of $x$ which are close enough to 0 is around $i^{-1+\varepsilon}$, see above, while the probability of a $\mathcal{N} (0,d^2/2)$ r.v.\ exceeding something around $\frac{d \sqrt{ \varepsilon}}{2} \sqrt{2 \log \log R^i}$ is roughly $e^{-\frac{\varepsilon}{2} \log \log R^i} \approx i^{-\varepsilon/2}$. Overall this gives a probability of $i^{-1+\varepsilon/2}$ that $\left| \sum_{k \in \Delta_i} f(\nu_k x) \right|$ is ``large'', which allows an application of the divergence Borel--Cantelli lemma.\\

Now we make this heuristic precise. Let $i\in\mathbb N$ and $h_i \in \mathbb N$ such that 
$$
\frac{1}{20 d 2^d M(i)} \leq 2^{-h_i} \leq \frac{1}{10 d 2^d M(i)},
$$
which implies that
$$
2^{-h_i} \geq \frac{1}{20 d 2^d \lceil i^{1-\varepsilon} \rceil} \geq i^{-1+5 \varepsilon/6}
$$
for sufficiently large $i$. If $i$ is sufficiently large, then for all $k \in \Delta_i$ we have $k \geq R^{i-1} \geq i \geq h_i$, so that by periodicity, for all $t > 0$, we have (recall that $\#\Delta_i = 
R^i$) 
\begin{align*}
&\lambda \left( \left\{ x \in [0,2^{-h_i}]\,:\,\sum_{k \in \Delta_i} \cos (2 \pi 2^k x) \geq t \right\} \right) \\
&=
2^{-h_i} \sum_{a=0}^{2^{h_i}-1} \lambda \left( \left\{ x \in \left[ \frac{a}{2^{h_i}}, \frac{a+1}{2^{h_i}} \right]\,:\,\sum_{k=1}^{R^i} \cos (2 \pi 2^k x) \geq t \right\} \right) \\
&=
2^{-h_i} \lambda \left( \left\{ x \in [0,1]\,:\,\sum_{k=1}^{R^i} \cos (2 \pi 2^k x) \geq t \right\} \right).
\end{align*}
Applying Gaposhkin's Berry--Esseen type estimate (see Lemma \ref{lemma_gaposhkin}) with all weights being equal, we obtain using $1 - \Phi(y) \geq 1 / (4 y e^{y^2/2})$ for $y \geq 4$ (see, e.g., \cite[Proposition 3]{szarekwerner1997}), that
\begin{eqnarray}
& & 2^{-h_i} \lambda \left( \left\{ x \in [0,1]\,:\, \sum_{k=1}^{R^i} \cos (2 \pi 2^k x) \geq \frac{\sqrt{\varepsilon}}{2} \sqrt{2 R^i \log \log R^i} \right\} \right) \nonumber\\
& \geq & 2^{-h_i}  \left( 1 - \Phi \left(\frac{\sqrt{\varepsilon}}{2} \sqrt{\log \log R^i} \right)  - \mathcal{O} (R^{-i/4}) \right) \nonumber\\
& \geq & \underbrace{2^{-h_i}}_{\geq i^{-1+5 \varepsilon/6} \text{ for sufficiently large $i$}} \underbrace{\left( \frac{e^{-\frac{\varepsilon}{4} \log \log R^i}}{4 \frac{\sqrt{\varepsilon}}{2} \sqrt{\log \log R^i}} - \mathcal{O} (R^{-i/4}) \right)}_{ \geq i^{-\varepsilon/3} \text{ for sufficiently large $i$}} \nonumber\\
& \geq & \frac{1}{i^{1-\varepsilon/2}}\label{lower_bound_p}
\end{eqnarray}
for sufficiently large $i \in \mathbb{N}$.\\

We need to show that the terms in \eqref{many_terms_2}, \eqref{many_terms_3} and \eqref{many_terms_4} do not make a relevant contribution when $x \in [0,2^{-h_i}]$.  Recall that by construction the smallest index $k$ in $\Delta_i$ is of size at least $R^{i-1}$, and that for sufficiently large $i\in\mathbb N$, we have $R^{i-1} \geq h_i$. We will work on short intervals of the form $\left[\frac{a}{2^{R^{i-1}}},\frac{a+1}{2^{R^{i-1}}} \right] \subset [0,2^{-h_i}]$ for some small integer $a$. Within such an interval, the function $\sum_{j=0}^{d-1} (\sin (\pi 2^j m x))^2$ is essentially constant. More precisely, writing
$$
s_{a,m,i} := \sum_{j=0}^{d-1} \left(\sin \left(\pi 2^j m \frac{a}{2^{R^{i-1}}}\right)\right)^2,
$$
by considering derivatives, we obtain the Lipschitz estimate
\begin{equation} \label{sami_discrete}
\left| \sum_{j=0}^{d-1} (\sin (\pi 2^j m x))^2 - s_{a,m,i} \right| \leq \frac{2^{d+1} \pi m}{2^{R^{i-1}}} \quad \text{for all $x \in \left[\frac{a}{2^{R^{i-1}}},\frac{a+1}{2^{R^{i-1}}} \right]$}.
\end{equation}

Furthermore, since we assumed that $\frac{a}{2^{R^{i-1}}} \in [0,2^{-h_i}]$, we have
\begin{eqnarray}
s_{a,m,i} & \leq & \sum_{j=0}^{d-1} \left(\sin \left(\pi 2^j m 2^{-h_i} \right)\right)^2 \nonumber\\
& \leq & \sum_{j=0}^{d-1} \left(\pi 2^j m 2^{-h_i} \right)^2 \nonumber\\
& \leq & \sum_{j=0}^{d-1} \left( \frac{\pi 2^j m}{10 d 2^d M(i)} \right)^2 \nonumber\\
& \leq & \frac{\pi^2}{100 d} \nonumber \\
& \leq &
\frac{1}{10 d}, \label{saim}
\end{eqnarray}
uniformly in $a$ and $m$. For $i \in\mathbb N$, we set
$$
S_{a,i} := \left(\sum_{m=1}^{M(i)}  \sum_{k \in \Delta_i^{(m)}} s_{a,m,i}^2 \right)^{1/2},
$$
and define
$$
\lambda_k := \frac{s_{a,m,i}}{S_{a,i}}, \qquad \text{for $k \in \Delta_i^{(m)}$.}
$$
Then, we clearly have
$$
\sum_{k \in \Delta_i} \lambda_k^2 = 1,
$$
and by \eqref{saim} it holds that
$$
S_{a,i} \leq \frac{R^{i/2}}{10 d}.
$$
Note that, using the estimates $S_{a,i} \geq \sqrt{\# \Delta_i^{(m)} s_{a,m,i}^2}$ and $\# \Delta_i^{(m)} \geq \frac{R^i}{\lceil i^{1- \varepsilon} \rceil } -1 $ for all $m \in \left \{ 1, \ldots, M(i)  \right \}$ by \eqref{mi_card}, we have
\begin{equation} \label{lambda_size}
\max_{k \in \Delta_i} \lambda_k = \max_{1 \leq m \leq M(i)} \frac{s_{a,m,i}}{S_{a,i}} \leq \max_{1 \leq m \leq M(i)} \frac{s_{a,m,i}}{\sqrt{\# \Delta_i^{(m)} s_{a,m,i}^2}} \leq \sqrt{\frac{\lceil i^{1-\varepsilon} \rceil }{R^{i} - \lceil i^{1- \varepsilon} \rceil}} \leq R^{-i/3}
\end{equation}
for sufficiently large $i \in\mathbb N$.\\

Thus, by periodicity, using Lemma \ref{lemma_gaposhkin} with the weights $\lambda_k$ as specified above, and using \eqref{lambda_size}, we have
\begin{eqnarray*}
& & 2^{R^{i-1}} \lambda \left( \left\{ x \in \left[\frac{a}{2^{R^{i-1}}},\frac{a+1}{2^{R^{i-1}}} \right] \,:\, \left| \sum_{m=1}^{M(i)} \sum_{k \in \Delta_i^{(m)}} s_{a,m,i} \cos (2 \pi 2^k x) \right| > \sqrt{2 R^i \log \log R^i} \right\} \right) \\
& = & 2^{R^{i-1}} \lambda \left( \left\{ x \in \left[\frac{a}{2^{R^{i-1}}},\frac{a+1}{2^{R^{i-1}}} \right] \,:\, \left| \sum_{k \in \Delta_i} \lambda_k \cos (2 \pi 2^k x) \right| > S_{a,i}^{-1}  \sqrt{2 R^i \log \log R^i} \right\} \right) \\
& \leq & 2^{R^{i-1}} \lambda \left( \left\{ x \in \left[\frac{a}{2^{R^{i-1}}},\frac{a+1}{2^{R^{i-1}}} \right] \,:\, \left| \sqrt{2} \sum_{k \in \Delta_i} \lambda_k \cos (2 \pi 2^k x) \right| >   \sqrt{400 d^2 \log \log R^i} \right\} \right) \\ 
& \leq & 1 - \Phi(\sqrt{400 d^2 \log \log R^i}) + c R^{-i/12} \\
& \leq & i^{-2}
\end{eqnarray*}
uniformly in $a$ for sufficiently large $i \in\mathbb N$ (the last estimate is very coarse, but the point is that our estimate leads to a convergent series), where $c > 0$ is an absolute constant. Note that \eqref{sami_discrete} implies
$$
\left| \sum_{m=1}^{M(i)} \sum_{j=0}^{d-1} (\sin (\pi 2^j m x))^2 \sum_{k \in \Delta_i^{(m)}} \cos (2 \pi 2^k x) - \sum_{m=1}^{M(i)} s_{a,m,i} \sum_{k \in \Delta_i^{(m)}} \cos (2 \pi 2^k x) \right| \leq R^i \frac{2^{d+1} \pi m}{2^{R^{i-1}}} \leq 1
$$
for sufficiently large $i \in\mathbb N$. After summing over all $a$ such that $\left[\frac{a}{2^{R^{i-1}}},\frac{a+1}{2^{R^{i-1}}} \right] \subset [0,2^{-h_i}]$, using the previous estimate and the triangle inequality, we finally arrive at
\begin{eqnarray*} 
& & \lambda \left( \left\{ x \in \left[0, 2^{-h_i} \right] :~ \left| \sum_{m=1}^{M(i)} \sum_{j=0}^{d-1} (\sin (\pi 2^j m x))^2 \sum_{k \in \Delta_i^{(m)}} \cos (2 \pi 2^k x)  \right| > \sqrt{2 R^i \log \log R^i}  + 1  \right\} \right) \nonumber\\
& & \leq 2^{-h_i} i^{-2} \leq i^{-2}
\end{eqnarray*}
for sufficiently large $i \in\mathbb N$, as a bound for the contribution of the term in line \eqref{many_terms_2}.\\

An analogous argument for the contribution of the term in line \eqref{many_terms_3} yields 
\begin{equation*}
\lambda \left( \left\{ x \in \big[0, 2^{-h_i} \big] \,:\, \left|  \sum_{m=1}^{M(i)} \sum_{j=0}^{d-1} \sin (2 \pi 2^j m x)   \sum_{k \in \Delta_i^{(m)}}  \sin (2 \pi 2^k x) \right| > \sqrt{2 R^i \log \log R^i}  + 1 \right\} \right) \leq \frac{1}{i^2},
\end{equation*}
where the relevant point for the argument is that the function $\sum_{j=0}^{d-1} \sin (2 \pi 2^j m x)$ is also very small in the interval $[0,2^{-h_i}]$ (and where we use a variant of Gaposhkin's Lemma \ref{lemma_gaposhkin} for sine instead of cosine).\\

Combining these two estimates with \eqref{rewritten_as}, \eqref{eim_error} and \eqref{lower_bound_p}, we obtain
\begin{eqnarray*}
\lambda \left( \left\{ x \in \left[0, 1 \right] \,:\, \left|  \sum_{k \in \Delta_i} f(\nu_k x) \right| > \left( \frac{d \sqrt{\varepsilon}}{2} - 2 \right)\sqrt{2 R^i \log \log R^i}  - 2 d^2 i - 2  \right\} \right) & \geq & \frac{1}{i^{1-\varepsilon/2}} - \frac{2}{i^2}
\end{eqnarray*}
for sufficiently large $i \in\mathbb N$. By \eqref{A_sup_set} and \eqref{A_set_nu} this implies
$$
\lambda (A_i) \geq i^{-1+\varepsilon/2} - 2 i^{-2}
$$
for all sufficiently large $i \in\mathbb N$. Thus, we have established \eqref{measure_diverges}, which completes the proof of Theorem \ref{th1}.

\section*{Acknowledgments}

CA was supported by the Austrian Science Fund (FWF), projects I-4945, I-5554, P-34763, and P-35322. LF was supported by the Austrian Science Fund (FWF), projects P-32405 and P-35322. JP was supported by the German Research Foundation (DFG) under project 516672205 and by the Austrian Science Fund (FWF) under project P-32405. 

\bibliography{LIL_Dio}
\bibliographystyle{abbrv}

\end{document}